\documentclass{amsart}[12pt]
\usepackage{xcolor}
\definecolor{linkblue}{rgb}{0,0.2,0.6}
\def\docpdftitle{Some singular curves in Mukai's model of $\overline{M}_7$}
\usepackage[
	hyperindex,
	pdftex,
	pdftitle={\docpdftitle},
	pdfdisplaydoctitle,
	pdfpagemode=UseNone,
	breaklinks=true,
	extension=pdf,
	bookmarks=false,
	plainpages=false,
	colorlinks,
	linkcolor=linkblue,
	citecolor=linkblue,
	urlcolor=linkblue,
	pdfmenubar=true,
	pdftoolbar=true,
	pdfpagelabels,
	pdfpagelayout=SinglePage,
	pdfview=Fit,
	pdfstartview=Fit
]{hyperref}
\usepackage[backrefs]{amsrefs}

\newtheorem{theorem}{Theorem}[section]

\newtheorem{proposition}[theorem]{Proposition}

\theoremstyle{definition}
\newtheorem{definition}[theorem]{Definition}

\newtheorem{question}[theorem]{Question}

\numberwithin{equation}{section}

\usepackage{tikz}

\newcommand{\codelink}[2]{\href{#1}{\sffamily{\texttt{Code {#2}}}}}
\newcommand{\Pro}{\mathbb{P}}
\newcommand{\Span}{\operatorname{Span}}

\newcommand{\Hilb}{\operatorname{Hilb}}

\newcommand{\Sym}{\operatorname{Sym}}

\newcommand{\SO}{\operatorname{SO}}
\newcommand{\OG}{\operatorname{OG}(5,10)}
\newcommand{\so}{\mathfrak{so}}

\newcommand{\Spin}{\operatorname{Spin}}
\newcommand{\Gr}{\operatorname{Gr}}

\newcommand{\even}{\operatorname{even}}
\newcommand{\odd}{\operatorname{odd}}

\newcommand{\git}{/\!\!/}
\newcommand{\M}{\overline{M}}
\newcommand{\mywedge}{\mathsf{\Lambda}}

\newcommand{\CCusp}{C_{\operatorname{cusp}}}
\newcommand{\CRib}{C_{\operatorname{rib}}}
\newcommand{\CNod}[1]{C_{\operatorname{nod},#1}}
\newcommand{\MCusp}{M_{\operatorname{cusp}}}
\newcommand{\MRib}{M_{\operatorname{rib}}}
\newcommand{\MNod}[1]{M_{\operatorname{nod},#1}}
\newcommand{\PCusp}{P_{\operatorname{cusp}}}
\newcommand{\PRib}{P_{\operatorname{rib}}}
\newcommand{\PNod}[1]{P_{\operatorname{nod},#1}}

\begin{document}

\title{Some singular curves in Mukai's model of $\M_7$}


\author{David Swinarski}
\address{Department of Mathematics\\ 
Fordham University\\ 
441 E Fordham Rd\\ 
Bronx, NY 10458}
\email{dswinarski@fordham.edu}


\date{}

\begin{abstract}
Mukai showed that the GIT quotient $\Gr(7,16) \git \Spin(10)$
is a birational model of the moduli space of Deligne-Mumford stable
genus 7 curves $\M_7$. The key observation is that a general smooth genus 7 curve can be realized
as the intersection of the orthogonal Grassmannian $\OG$ in
$\mathbb{P}^{15}$ with a six-dimensional projective linear
subspace. What objects appear on the boundary of Mukai's model? As a first step in this
study, computer calculations in \texttt{Macaulay2}, \texttt{Magma}, and \texttt{Sage}
are used to find and analyze linear spaces yielding three examples of singular curves: a 7-cuspidal curve, the
balanced ribbon of genus 7, and a family of genus 7 reducible nodal
curves. $\Spin(10)$-semistability is established by constructing and
evaluating an invariant polynomial.
\end{abstract}

\maketitle

\section{Introduction}

In 1995 Mukai showed that the GIT quotient $\Gr(7,16) \git \Spin(10)$
is a birational model of the moduli space of Deligne-Mumford stable
genus 7 curves $\M_7$. We briefly recall this correspondence.

For a general curve of genus $g\geq 3$, the canonical ideal $I$ is
generated by $\binom{g-2}{2}$ quadrics. Thus, when $g=7$, 10 quadrics
in $\Pro^6$ are required.

Mukai showed that for a smooth genus 7 curve with no $g_2^1$, $g_3^1$,
or $g_4^1$, the multiplication map  $\Sym^2(I_2) \rightarrow I_4$ has
a one-dimensional kernel. Let $Q$ be a generator of the kernel. Then
$(I_2,Q)$ is a 10-dimensional quadratic vector space.

Let $f_0,\ldots,f_9 \in k[x_0,\ldots,x_6]$ generate $I_2$. For each $p
\in C$, the row space of the Jacobian matrix at $p$
\[
\left[ \frac{\partial f_j}{\partial x_i}(p) \right]_{i=0,\ldots,6}^{j=0,\ldots,9}
\]
is a Lagrangian of $(I_2,Q)$, which Mukai denotes $W_{p}^{\perp}$.

Let $\OG$ denote the ten-dimensional orthogonal Grassmannian
parametrizing Lagrangian subspaces of $(I_2,Q)$. $\OG$ has a natural
embedding in $\Pro^{15}$ by mapping a Lagrangian to its half spinor.

\begin{theorem}[Mukai, 1995]
  Let $C$ be a smooth genus 7 curve with no $g^{1}_{2}$, $g^{1}_{3}$,
  or $g^{1}_{4}$.
    \begin{enumerate}
    \item The map
      \[
  \begin{array}{cccccc}
    \rho: & C & \rightarrow & \OG & \rightarrow & \Pro^{15}\\
     & p & \mapsto & [W_p^{\perp}]
  \end{array}
  \]
is an embedding of $C$.
\item The image $\rho(C)$ is the intersection $(P \cap \OG)$ of a 6-dimensional projective linear
subspace $P \subset \Pro^{15}$ with the orthogonal Grassmannian, and $C$ is canonically embedded in $P$.
\item $\Gr(7,16) \git \Spin(10)$ is a birational model of $\M_7$.
\end{enumerate}
\end{theorem}
See  \cite{Mukai}*{Theorem 0.4 and Prop.~5.2}.


Let $S^{+}$ be the half-spin representation of $\Spin(10)$. We have
$\dim S^{+} = 16$. A character calculation shows that there exist $\Spin(10)$-invariant
polynomials on $\mywedge^7 S^{+}$; see \codelink{https://faculty.fordham.edu/dswinarski/MukaiModelOfM7/v1/1.1.htm}{1.1}. It follows that a general point of
$\Gr(7,16)$ is $\Spin(10)$-semistable.  Also, Farkas and Verra give
some  $\Spin(10)$-semistability results for the related quotient $\Hilb(\OG) \git \Spin(10)$ in \cite{FV}.

However, several questions remain open. Is every smooth genus 7 curve with no $g^{1}_{2}$, $g^{1}_{3}$,
or $g^{1}_{4}$ $\Spin(10)$-semistable? Which schemes occur as intersections $P \cap
\OG$, and when is $[P]$ $\Spin(10)$-semistable?

As a first step, we study three examples of singular curves.
\begin{center}
\begin{tabular}{ll}
Example 1: & $\CCusp$, the $7$-cuspidal curve with heptagonal symmetry\\
Example 2: & $\CRib$, the balanced ribbon of genus 7\\
  Example 3: & $\CNod{t}$, a family of reducible nodal curves degenerating to three \\
  & trivalent graph curves (in the sense of Bayer and Eisenbud)
\end{tabular}
\end{center}
The rationale for these choices is as follows. The orthogonal
Grassmannian $\OG \subset \Pro^{15}$ has the following Betti
table, displayed following \texttt{Macaulay2}'s conventions. See \codelink{https://faculty.fordham.edu/dswinarski/MukaiModelOfM7/v1/1.2.htm}{1.2}.
\begin{center}
\begin{verbatim}
                           0  1  2  3  4 5
                    total: 1 10 16 16 10 1
                        0: 1  .  .  .  . .
                        1: . 10 16  .  . .
                        2: .  .  . 16 10 .
                        3: .  .  .  .  . 1
\end{verbatim}
\end{center}
A Betti table with at most one nonzero entry per column is called
\textit{pure}. See \cite{ES}.

If a linear section $P \cap \OG$ is one-dimensional, it must also have this Betti table. Curves with pure Betti tables have been the subject
of much study for several years in connection with Green's
Conjecture. $g$-cuspidal curves, ribbons, and graph curves were
proposed as candidates for proving Green's
Conjecture for a generic curve \cite{E}. (This strategy was recently completed for 
$g$-cuspidal curves by the results of \cite{AFPRW} and for ribbons by
the results of \cite{RS}.)

Specific $g$-cuspidal curves, ribbons, and graph curves with 
automorphisms have also been used to study the
Hassett-Keel program for $(\M_g,\Delta)$. One of Hassett and Keel's
conjectures was that the canonical model of $\M_g$ could be
constructed by variation of GIT applied to quotients of spaces
parametrizing syzygies of curves. The GIT semistability of the canonically
embedded balanced ribbon was established in \cite{AFS} for finite
Hilbert stability and in \cite{DFS2016} for first syzygies. The $7$-cuspidal curve and the graph
curve $\CNod{0}$ studied here also have GIT semistable second Hilbert
points and first syzygies (Swinarski, unpublished). Since these three
examples of singular curves
appear in the model of $\M_7$ given by first syzygies, it was natural
to ask whether they also appear in Mukai's model of $\M_7$.

\subsection{Outline of the paper} In Section \ref{sec:background}, we
recall the notation of Mukai's construction. In Sections
\ref{sec:cusp} and \ref{sec:ribbon}
we obtain the $7$-cuspidal curve with heptagonal
symmetry and the balanced genus 7 ribbon as intersections $P \cap \OG$
for some explicit $P \in
\Gr(7,16)$. In Section \ref{sec:graph curves} we describe a 1-parameter family
of reducible nodal curves and obtain a general member of this family
as the intersection $P \cap \OG$ for some explicit $P \in
\Gr(7,16)$. We also study the limits as this family degenerates in the
Hilbert scheme $\Hilb(\mathbb{P}^{15},12t-6)$ and in the Grassmannian
$\Gr(7,16)$ and show that two of these limits are GIT-unstable.

In Section \ref{sec:constructing invariants} we describe how to
construct a $\Spin(10)$-invariant polynomial $F_{5\omega_1} \in (\Sym^4\mywedge^7 S^{+})^{\Spin(10)}$. 
Finally, we evaluate $F_{5\omega_1}$ on these three examples to deduce
$\Spin(10)$-semistability for the $7$-cuspidal curve with heptagonal
symmetry, the balanced genus 7 ribbon, and the general member of the 
family of reducible nodal curves.

\subsection*{Software and code links} This project relies heavily on
calculations in \texttt{Macaulay2}, \texttt{Magma}, and
\texttt{Sage} \cites{M2,Magma,Sage}. In this document, we report the inputs to these
calculations and describe the results. On the author's webpage
\cite{Code}, we have posted transcripts of interactive sessions
for the shorter calculations and the input
and output files used for the lengthier calculations.  We cite each calculation in
the text of this document using a phrase of the form ``see
\codelink{https://faculty.fordham.edu/dswinarski/MukaiModelOfM7/v1/index.htm}{x.y}''
which includes a link to the relevant calculation.

\subsection*{Acknowledgements} It is a pleasure to thank 
Patricio Gallardo, Jesus Martinez-Garcia, Han-Bom Moon, and Ian
Morrison for several helpful discussions related to this work. This
work is a sequel to a project begun by the AIM Square ``Computational aspects of GIT with a view of moduli spaces''
that met between 2018--2020
consisting of Gallardo, Martinez-Garcia, Moon,
and the author.

\section{Background: Mukai's construction} \label{sec:background}

Let $V$ be a $2n$-dimensional vector space over $\mathbb{C}$. (Note:
Mukai's results hold over an algebraically closed field of any
characteristic. We will state our results only for $\mathbb{C}$, but
it seems likely that some of them may generalize to positive
characteristic as well.) Let $Q$ be a
full rank quadratic form on $V$. Following Chevalley and Mukai's
conventions in \cites{Chevalley,Mukai}, let
$B(x,y) = Q(x+y)-Q(x)-Q(y)$. (Note: Fulton and Harris
use a different convention in \cite{FH}.) Then $B(x,x) = 2Q(x)$. Let $C(Q)$ be the Clifford algebra satisfying  $v \cdot w + w \cdot v = B(v,w) \cdot 1.$           

Let $U_0$ and $U_{\infty}$ be two complementary Lagrangians, and let $S^{+} = \mywedge^{\even} U_{\infty}$, $S^{-} = \mywedge^{\odd} U_{\infty}$.

Let $e_{-1},\ldots,e_{-n}$ be a basis of $U_0$, and let
$e_{1},\ldots,e_{n}$ be a basis of $U_{\infty}$. $e_{-i}$ acts on
$\mywedge  U_{\infty}$ as the contraction of $e_i$, and $e_i$ acts on
$\mywedge  U_{\infty}$  as wedging on the left by $e_i$.  Extending
these actions by linearity yields an endomorphism $\varphi_v$ for any $v \in
V$.

For each subset $I = \{i_1,\ldots,i_k\} \subset \{1,\ldots,5\}$ with
$k$ even and $i_1 < \ldots< i_k$, let $e_I = e_{i_1} \wedge \cdots
\wedge e_{i_k}$. This gives a basis of $S^{+}$. Let $x_I$ be the
corresponding coordinates on $\Pro(S^{+})$.

Let $U$ be a Lagrangian of $(V,Q)$. The half spinor $s_U$ of $U$ is an
element of $S^{+} \cup S^{-}$ satisfying $\varphi_u(s_U) = 0$ for all
$u \in U$.

We use two
  approaches to compute half spinors.

  \textit{Approach 1:} Suppose \( U \cap U_{\infty} = \{0\} \). (This is the
  generic case.) Then we can find a basis of \(U\) of the form \(u_i = e_{-i}
- \sum_{j=1}^{5} a_{ij} e_{j}.\) The coefficients \(a_{ij}\) yield a
\(5 \times 5\) skew-symmetric matrix \(A\). In the proof of
\cite{Mukai}*{Prop.~1.5}, Mukai gives a formula for \(s_U\) in terms
of the Pfaffians of minors \(A\). Specifically, let $A_I$ denote the
minor of $A$ obtained by selecting the rows and columns indexed by
$I$. Then the coordinate $x_I$ in $[s_U]$ is given by $\operatorname{Pf}(A_I)$.

\textit{Approach 2:} For any Lagrangian $U$, we may compute the operators \(\varphi_u\) for
    a basis of \(U\) and intersect their kernels to obtain a suitable
    \(s_U\).

Approach 2 applies to any Lagrangian \(U\), but it is typically slower than Approach 1, so we only use Approach 2 when \( \dim(U \cap U_{\infty}) > 0. \)

Mukai gives the following equations of the orthogonal Grassmannian
$\OG \subset \Pro(S^{+})$ in \cite{Mukai}*{(0.1)}:
\[
\begin{array}{l}
x_{0} x_{2345}-x_{23} x_{45}+x_{24} x_{35}-x_{25} x_{34},\\
x_{12} x_{1345}-x_{13} x_{1245}+x_{14} x_{1235}-x_{15} x_{1234},\\
x_{0} x_{1345}-x_{13} x_{45}+x_{14} x_{35}-x_{15} x_{34},\\
x_{12} x_{2345}-x_{23} x_{1245}+x_{24} x_{1235}-x_{25} x_{1234},\\
x_{0} x_{1245}-x_{12} x_{45}+x_{14} x_{25}-x_{15} x_{24},\\
x_{13} x_{2345}-x_{23} x_{1345}+x_{34} x_{1235}-x_{35} x_{1234},\\
x_{0} x_{1235}-x_{12} x_{35}+x_{13} x_{25}-x_{15} x_{23},\\
x_{14} x_{2345}-x_{24} x_{1345}+x_{34} x_{1245}-x_{45} x_{1234},\\
x_{0} x_{1234}-x_{12} x_{34}+x_{13} x_{24}-x_{14} x_{23},\\
x_{15} x_{2345}-x_{25} x_{1345}+x_{35} x_{1245}-x_{45} x_{1235}
\end{array}
\]

\section{The $7$-cuspidal curve with heptagonal
  symmetry}\label{sec:cusp}

Canonically embedded \(g\)-cuspidal curves
can be obtained as hyperplane sections of the \emph{tangent
  developable} of the rational normal curve. See \cites{E, AFPRW, EL}
for more details.

For $g=7$, the tangent developable in
\(\mathbb{P}^{7}\) is parametrized by mapping
$(s,t,u,v)$ to 
\[
[7 s^6 u : 6 s^5 t u+s^6 v : 5 s^4 t^2 u+2 s^5 t v :
4 s^3 t^3 u+3 s^4 t^2 v : 3 s^2 t^4 u+4 s^3 t^3 v : 2 s t^5 u+5 s^2
t^4 v :  t^6 u+6 s t^5 v : 7 t^6 v].
\]

We eliminate the parameters to obtain equations of the tangent
developable in $k[y_0,\ldots,y_7]$; see \codelink{https://faculty.fordham.edu/dswinarski/MukaiModelOfM7/v1/3.1.htm}{3.1}.  Then, by taking \(y_7=y_0\) we
get equations of a rational curve $\CCusp$ with seven
cusps. The cusps occur where the hyperplane section meets the
diagonal, that is,  at the seventh roots of unity $(s/t)^7=1$. Hence this curve has the dihedral group
\(D_7\) of order \(14\) as its automorphism group; see \codelink{https://faculty.fordham.edu/dswinarski/MukaiModelOfM7/v1/3.2.htm}{3.2}.

This yields the following 10 quadrics generating $I_2$.
\[
\begin{array}{rclcrcl}
 f_0 &= &3 y_5^2-4 y_4 y_6+y_3 y_0 & \qquad &  f_5 &= &y_3 y_4-2 y_1 y_6+y_0 y_0\\
 f_1 &= &2 y_4 y_5-3 y_3 y_6+y_2 y_0 & \qquad& f_6 &= &5 y_2 y_4-8 y_1 y_5+3 y_0 y_6\\
 f_2 &= &5 y_3 y_5-8 y_2 y_6+3 y_1 y_0 & \qquad& f_7 &=& 5 y_3^2-9 y_1 y_5+4 y_0 y_6\\
 f_3 &= &3 y_2 y_5-5 y_1 y_6+2 y_0 y_0 &\qquad & f_8 &= &2 y_2 y_3-3 y_1y_4+y_0 y_5 \\
 f_4 &= &5 y_4^2-9 y_2 y_6+4 y_1 y_0 &\qquad & f_9 &= &3 y_2^2-4 y_1 y_3+y_0 y_4
\end{array}
\]
The automorphisms are given by the maps $y_i \mapsto \zeta_7^i
y_i$ and $[y_0:y_1:y_2:y_3:y_4:y_5:y_6] \mapsto [y_0:y_6:y_5:y_4:y_3:y_2:y_1]$.

Next, we compute \( \mathrm{ker}(\mathrm{Sym}^2(I_2) \rightarrow I_4)\)
in \texttt{Macaulay2}, and find that these quadrics satisfy the following quadratic form.
\[-f_{3}^{2}+\frac{9}{2}f_{3}f_{5}-5\,f_{5}^{2}-\frac{3}{10}f_{2}f_{6}+\frac{1}{5}f_{4}f_{7}-\frac{3}{2}f_{1}f_{8}+f_{0}f_{9}
  = 0.
\]

We change the basis of $I_2$ as follows.
\[
\begin{array}{rclcrcl}
 g_0 &= &-10 f_0 & \qquad &  g_5 &= &f_9\\
 g_1 &= &15 f_1 & \qquad & g_6 &= &f_8\\
 g_2 &= &3 f_2 & \qquad & g_7 &=& f_6\\
 g_3 &= &-2f_4 & \qquad &g_8 &= &f_7\\
 g_4 &= &-10f_3+25 f_5 & \qquad & g_9 &= &-f_3+2f_5
\end{array}
\]
Then $\sum_{i=0}^{4} g_i g_{i+5} = 0$; see \codelink{https://faculty.fordham.edu/dswinarski/MukaiModelOfM7/v1/3.3.htm}{3.3}.

Next, we arbitrarily choose eight smooth points $p_0,\ldots,p_7$ in general position on $\CCusp$. These points are given by the following values of
\( (s,t,u,v)\) under the parametrization shown above:
$ (-1,1,1,1)$, $(1,2,64,1)$, $(2,1,1,64)$, $(1,3,729,1)$, $(3,1,1,729)$, $(1,-2,64,1)$,
$(-2,1,1,64)$, $(1,-3,729,1)$. (Seven points are sufficient to determine the linear space $\PCusp$; the eighth point will be
used to prove that the map $\rho: \CCusp \rightarrow \PCusp$ is an embedding.)

To each point on $\CCusp$ we associate the Lagrangian that Mukai denotes
\(W_p^{\perp}\), which we interpret as the row space of the
Jacobian matrix \( \left[\frac{\partial g_j}{\partial x_i}(p)\right]\).

Next, we need to choose a pair of complementary Lagrangians $U_0$ and $U_{\infty}$. Every Lagrangian will have even-dimensional intersection with one of these and odd-dimensional intersection with the other. Mukai assumes that \(U_0 \) and
\(U_\infty\) are chosen so that \(W_{p}^{\perp}\) has even-dimensional
intersection with \(U_\infty\). We choose \(U_0  = \Span \{
g_0,\ldots,g_4\} \) and
\(U_\infty = \Span \{g_5,\ldots,g_{9}\}\) and check that our choices satisfy this
property.

Next, we compute the half spinors $s_i$ of the Lagrangians $W_{p}^{\perp}$
associated to the points $p_i$. We find that $s_0,\ldots,s_7$ span the 7-dimensional
vector space given by the row space of the following matrix.
\begin{equation} \label{MCusp}
\MCusp= \left[\begin{array}{rrrrrrrrrrrrrrrr}
0 & 0 & 0 & 0 & -\frac{3}{5} & 1 & 0 & 0 & 0 & 0 & 0 & 0 & 0 & 0 & 0 & 0\\
0 & 0 & 0 & 0 & 0 & 0 & 0 & \frac{1}{5} & 1 & 0 & 0 & 0 & 0 & 0 & 0 & 0\\
0 & 0 & 0 & 0 & 0 & 0 & 0 & 0 & 0 & \frac{3}{4} & 1 & 0 & 0 & 0 & 0 & 0\\
30 & 0 & 0 & 0 & 0 & 0 & 0 & 0 & 0 & 0 & 0 & 1 & 0 & 0 & 0 & 0\\
0 & 0 & 0 & 0 & 0 & 0 & 0 & 0 & 0 & 0 & 0 & 0 & \frac{8}{9} & 1 & 0 & 0\\
0 & -2 & 0 & 0 & 0 & 0 & 0 & 0 & 0 & 0 & 0 & 0 & 0 & 0 & 1 & 0\\
0 & 0 & 0 & -\frac{15}{2} & 0 & 0 & 0 & 0 & 0 & 0 & 0 & 0 & 0 & 0 & 0 & 1
\end{array} \right]
\end{equation}

Let $\PCusp = \mathbb{P}(\operatorname{RowSpace} \MCusp)$. We check that
$\PCusp \cap \OG \cong \CCusp$. To do this, we fix an isomorphism
$h: \PCusp \rightarrow \Pro^6$, then compute the unique
element of $\operatorname{PGL}(7)$ mapping $p_i$ to $h(s_i)$ for
$i=0,\ldots,7$, and check that this maps $\CCusp$ to $h(\PCusp \cap
\OG)$; see \codelink{https://faculty.fordham.edu/dswinarski/MukaiModelOfM7/v1/3.4.htm}{3.4}.

These calculations establish the following proposition.

\begin{proposition} \label{prop:PCusp}
  Let $\CCusp$ be the $7$-cuspidal curve with heptagonal
  symmetry. Then $\rho: \CCusp^{\operatorname{sm}} \rightarrow
  \mathbb{P}^{15}$ extends to an embedding, and $\rho(\CCusp) = \PCusp  \cap \OG$, where $\PCusp = \mathbb{P}(\operatorname{RowSpace} \MCusp)$.
\end{proposition}

\section{The balanced genus 7 ribbon}\label{sec:ribbon}
 Ribbons are dimension 1, generically nonreduced schemes that are double
  structures on the underlying reduced curve. Bayer and Eisenbud write
  in their seminal paper on ribbons that ribbons are limits of the
  canonical models of smooth curves as they degenerate to a
  hyperelliptic curve \cite{BE}. A longstanding prediction of  the Hassett-Keel program for  $(\overline{M}_g,\Delta)$ is that the locus of hyperelliptic curves is flipped to the ribbon locus.

We consider a specific example. In each odd genus $g=2k+1$ with $g
\geq 5$ there is a ribbon called the \emph{balanced ribbon}, which is
characterized by having a $ \mathbb{G}_m$-action with weights $
-k,\ldots,+k$ as well as an involution interchanging the positive and
negative weight spaces. Equations of the canonically embedded
genus 7 balanced ribbon can be obtained using \cite{DFS2014}*{Cor.~4.8}.
\[
\begin{array}{rclcrcl}
f_0 &=& y_2 y_3-2 y_1 y_4+y_0 y_5 & \qquad & f_5 &=  -y_1 y_2+y_0 y_3\\
f_1 &=&  y_2 y_4-2 y_1 y_5+y_0 y_6 & \qquad & f_6 &=  -y_2^2+y_1 y_3\\
f_2 &=&  y_3^2-2 y_2 y_4+y_1 y_5 & \qquad & f_7 &=  -y_4^2+y_3 y_5\\
f_3 &=& y_3 y_4-2 y_2 y_5+y_1 y_6 & \qquad & f_8 &=  -y_4 y_5+y_3 y_6,\\
f_4 &=&  -y_1^2+y_0 y_2 & \qquad & f_9 &=  -y_5^2+y_4 y_6\\
\end{array}
\]

The variables $y_0,\ldots,y_6$ have weights $-3,\ldots,3$, and the
involution acts by sending $y_0,\ldots,y_6$ to $y_6,\ldots,y_0$.

Next, we compute $ \mathrm{ker}(\mathrm{Sym}^2(I_2) \rightarrow I_4)$
in \texttt{Macaulay2}, and find that these quadrics satisfy the following quadratic form.
\[
\frac{1}{2}f_1f_2 - \frac{1}{2}f_0f_3+f_6f_7-\frac{1}{2}f_5f_8 + f_4f_9  = 0.
\]
We reorder the quadrics so that the $\mathbb{G}_m$ weights are $-4,-3,-2,-1,0,4,3,2,1,0$, and scale to make the coefficients of the quadratic form 1.
\[
\begin{array}{rclcrcl}
 g_0 &= &2 f_4 & \qquad &  g_5 &= &f_9 \\
 g_1 &= &- f_5 & \qquad & g_6 &= &f_8 \\
 g_2 &= &2 f_6 & \qquad &   g_7 &= &f_7 \\
 g_3 &= &-f_0 & \qquad & g_8 &= &f_3 \\
 g_4 &= & f_1 & \qquad & g_9 &= & f_2
\end{array}
\]
Then $\sum_{i=0}^{4} g_i g_{i+5} = 0$; see \codelink{https://faculty.fordham.edu/dswinarski/MukaiModelOfM7/v1/4.1.htm}{4.1}.

Next, we compute the spin representation of the automorphism group of the
  balanced ribbon. Let $e_{-1},\ldots,e_{-5}$ be $g_0,\ldots,g_4$, and let $e_1,\ldots,e_5$ be $g_5,\ldots,g_9$. Then $ \mathbb{G}_m$ acts on the basis $e_{-1},\ldots,e_5$ by 
  \[
\operatorname{Diag}(t^{-4}, t^{-3}, t^{-2}, t^{-1}, 1,  t^{4}, t^{3}, t^{2}, t^{1}, 1),
  \]
  and the involution acts on this basis by
\[
  \begin{array}{rclcrcl}
  e_{-1} & \mapsto & \frac{1}{2}e_1 & \qquad & e_{-4} & \mapsto & -e_4 \\  
    e_{-2} & \mapsto &-e_2 & \qquad & e_{-5} & \mapsto &  e_{-5} \\
    e_{-3} & \mapsto & \frac{1}{2}e_3 & \qquad &  e_{5} & \mapsto &e_{5} 
\end{array}
\]
  
To lift these elements to $\mathrm{Spin}(Q)$, we factor them as a product of
reflections, lift each reflection to the Clifford algebra, and scale. We find that the $\mathbb{G}_m$ action lifts to the following two elements in $\Spin(10)$.
  \[
  \pm t^5 \prod_{j=1}^{4} ( e_{-j} + e_{j}) (e_{-j} + t^{j-5} e_{j}).
  \]
  The involution lifts to the elements
  \[
\pm    2 (e_{-4}+e_4) (e_{-3}-\frac{1}{2}e_3) (e_{-2}+e_2) (e_{-1}-\frac{1}{2}e_1).
\]
Thus, the $\mathbb{G}_m$ action on the basis
\[ 1, e_{12}, e_{13},
e_{14}, e_{15}, e_{23}, e_{24}, e_{25}, e_{34}, e_{35}, e_{45},
e_{1234}, e_{1235}, e_{1245}, e_{1345}, e_{2345}
\]
of $S^{+}$ is given by
\[\operatorname{Diag}( t^{-5}, t^{2}, t, 1,  t^{-1}, 1, t^{-1}, t^{-2},t^{-2}, t^{-3}, t^{-4}, t^5, t^4, t^3, t^2, t)
\]
and the involution acts on this basis as follows.
\[
  \begin{array}{rclcrcl}
  1 & \mapsto &\frac{1}{2}e_{1234} & \qquad & e_{15} & \mapsto & e_{2345}\\
  e_{12} & \mapsto & e_{34} & \qquad &e_{25} & \mapsto & \frac{1}{2} e_{1345} \\
  e_{13} & \mapsto & 2e_{24} & \qquad &e_{35} & \mapsto & e_{1245}\\
  e_{14} & \mapsto & e_{23} & \qquad &e_{45} & \mapsto & \frac{1}{2}e_{1235}
\end{array}
\]
See \codelink{https://faculty.fordham.edu/dswinarski/MukaiModelOfM7/v1/4.2.htm}{4.2}.

We seek a six-dimensional projective linear subspace $\PRib$ such that $\PRib \cap \OG \cong \CRib$. We know the weights of the $\mathbb{G}_m$ action on the canonically
embedded balanced ribbon, and that the involution swaps positive and negative weight spaces. We use this to narrow down the search for $\PRib$.

The $\mathbb{G}_m$ weights on the ribbon are $-3,-2,-1,0,1,2,3$, while the $\mathbb{G}_m$ weights on $\mathbb{P}(S^{+})$ are (in increasing order) $-5,-4,-3,-2,-2,-1,-1,0,0,1,1,2,2,3,4,5$. By comparing these two lists, we see that we must kill the $\pm 5$ and $\pm 4$ weight spaces; retain the $\pm 3$ weight spaces; and select a multiplicity 1 submodule of the multiplicity 2 weight spaces for weights $\pm 2$, $\pm 1$, and 0.

The $\pm 4$ and $\pm 5$ weight spaces are spanned by $ x_{45}$,
$x_{1235}$, $x_0$, and $x_{1234}$.  Thus we set $x_{45} = x_{1235}=x_0=x_{1234}=0$. This gives us 4 of the 9 hyperplanes we seek to define the linear space $P$.

Next, consider the weight 0 space. This is spanned by $x_{14}$ and $x_{23}$. The involution acts on this subspace as $ x_{14} \mapsto x_{23}$. Since the involution is trivial on the weight 0 space for the balanced ribbon, we set $ x_{14}=x_{23}=0.$. This gives a fifth hyperplane.

  Next, consider the weight $\pm 1$ space. It is a multiplicity two module with respect to the automorphism group. A general submodule can be written in the form $ \mathrm{Span}\langle c_1 x_{13}+c_{2}x_{2345}, \frac{1}{2}c_{1}x_{24} + c_2 x_{15}\rangle$ for some constants $c_1$ and $c_2$. 
  Assume $c_1\neq 0$. Then we can scale these to obtain two more hyperplanes $x_{13} +c_2 x_{2345} = 0$ and $ \frac{1}{2}x_{24} + c_2 x_{15}=0.$

  Similarly, the weight $\pm 2$ space is a multiplicity two module with respect to the automorphism group. A general submodule can be written in the form $ \mathrm{Span}\langle c_3 x_{12}+c_{4}x_{1345}, c_{3}x_{34} + \frac{1}{2}c_4 x_{25}\rangle$ for some constants $c_3$ and $c_4$. 
  Assume $c_3\neq 0$. Then we can scale these to obtain the hyperplanes $x_{12} +c_4 x_{1345} = 0$ and $ x_{34} + \frac{1}{2}c_4 x_{25}=0.$

  We have thus found nine linearly independent hyperplanes with two
  unknown parameters $c_2$ and $c_4$. For each pair $c_2,c_4$, let $P_{c_2,c_4}$ be the
  six-dimensional projective linear subspace defined by these
  hyperplanes. For any values of $c_2$
  and $c_4$, the intersection of $P_{c_2,c_4}$ with the orthogonal
  Grassmannian yields a scheme with a $\mathbb{G}_m$-action with
  weights $-3,-2,-1,0,1,2,3$ and
 an  involution interchanging the positive and negative weight
  spaces. Are there any values of $c_2$ and $c_4$ that yield the balanced ribbon?

Next, choose seven of the variables \(x_I\) with weights
\(-3,-2,-1,0,1,2,3\) to use as variables on $\PRib \cong
\mathbb{P}^6$. Here we used $ y_0 = x_{1245}, y_1 = x_{1345}, y_2 = x_{2345}, y_3 =x_{14}, y_4=x_{15}, y_5 = \frac{1}{2}x_{25}, y_6 = x_{35}.$ (The choice $y_5 = \frac{1}{2} x_{15}$ is because the involution on the balanced ribbon swaps the \(\pm 2\) weight spaces, and this is the basis that has the desired action.)

Substituting the nine hyperplanes found above into Mukai's equations
for the orthogonal Grassmannian yields the following quadrics.
    \[
  \begin{array}{l}
  2 c_4 y_1^2-2 c_2 y_0 y_2, \\
  -c_4 y_5^2+c_2 y_4 y_6,\\
  c_4 y_1 y_2+y_0 y_3, \\
  -c_4 y_4 y_5-y_3 y_6, \\
  2 c_2 y_2^2+2 y_1 y_3, \\
  -c_2 y_4^2-y_3 y_5, \\
  -y_2 y_3-2 c_2 y_1 y_4+c_4 y_0 y_5, \\
  y_3 y_4+2 c_2 y_2 y_5-c_4 y_1 y_6, \\
  -y_3^2+2 c_2^2 y_2 y_4-c_4^2 y_1 y_5, \\
  y_2 y_4-2 y_1 y_5+y_0 y_6
 \end{array}
\]
See \codelink{https://faculty.fordham.edu/dswinarski/MukaiModelOfM7/v1/4.3.htm}{4.3}.

  Careful inspection reveals that with $c_2 = -1$ and $c_4 = -1$, each quadric on our list is a nonzero constant multiple of one of the balanced ribbon equations.

These calculations establish the following proposition.

\begin{proposition} \label{prop:PRib}
  Let $\CRib$ be the genus 7 balanced ribbon. Then 
  \[
  \CRib \cong \PRib  \cap \OG
\]
where
\begin{equation} \label{MRib}
\MRib= \left[\begin{array}{rrrrrrrrrrrrrrrr}
0 & 0 & 0 & 1 & 0 & 1 & 0 & 0 & 0 & 0 & 0 & 0 & 0 & 0 & 0 & 0\\
0 & 0 & 0 & 0 & \frac{1}{2} & 0 & 1 & 0 & 0 & 0 & 0 & 0 & 0 & 0 & 0 & 0\\
0 & 0 & 0 & 0 & 0 & 0 & 0 & 2 & 1 & 0 & 0 & 0 & 0 & 0 & 0 & 0\\
0 & 0 & 0 & 0 & 0 & 0 & 0 & 0 & 0 & 1 & 0 & 0 & 0 & 0 & 0 & 0\\
0 & 0 & 0 & 0 & 0 & 0 & 0 & 0 & 0 & 0 & 0 & 0 & 0 & 1 & 0 & 0\\
0 & 1 & 0 & 0 & 0 & 0 & 0 & 0 & 0 & 0 & 0 & 0 & 0 & 0 & 1 & 0\\
0 & 0 & 1 & 0 & 0 & 0 & 0 & 0 & 0 & 0 & 0 & 0 & 0 & 0 & 0 & 1
\end{array} \right]
\end{equation}

and $\PRib = \mathbb{P}(\operatorname{RowSpace} \MRib)$
\end{proposition}

\section{A family of reducible nodal curves} \label{sec:graph curves}
Next, we study a family of reducible nodal curves. This family is a one-dimensional stratum in the boundary of $\M_7$, also known as an F-curve. 

  This family is constructed as follows. Let $G$ be the graph on 11 vertices $0,1,2,34,5,6,7,8,9,10,11$ with edges 0-1, 0-8, 0-9, 1-2, 1-10, 2-34, 2-11, 34-9, 34-5, 34-10, 5-6, 5-11,
  6-7, 6-9, 7-8, 7-10, and 8-11. We present two different views of
  this graph. See Figures \ref{K33Figure} and \ref{TrivalentSpecializationsFigure}.

\begin{figure}[h]
  \begin{center}
    \caption{The graph $G$}
        \label{K33Figure}    
\begin{tikzpicture}[scale=0.5]
  \draw (0,0)--(1,0); 
  \draw (0,0) arc (180:360:10 and 2);
  \draw (0,0)--(20,0); 
  \draw (0,0)--(2,5); 
  \draw (2,0)--(4,0); 
  \draw (2,0)--(10,5); 
  \draw (4,0) -- (9,0); 
  \draw (4,0)--(18,5); 
  \draw (9,0)--(2,5); 
  \draw (9,0)--(12,0); 
  \draw (9,0)--(10,5); 
  \draw (12,0) -- (16,0); 
  \draw (12,0)--(18,5); 
  \draw (16,0)--(18,0); 
  \draw (16,0)--(2,5); 
  \draw (18,0)--(20,0); 
  \draw (18,0)--(10,5); 
  \draw (20,0)--(18,5); 
  \filldraw[white] (0,0) circle (0.5);
  \filldraw[white] (2,0) circle (0.5);
  \filldraw[white] (4,0) circle (0.5);
  \filldraw[white] (9,0) circle (0.5);
  \filldraw[white] (12,0) circle (0.5);
  \filldraw[white] (16,0) circle (0.5);
  \filldraw[white] (18,0) circle (0.5);
  \filldraw[white] (20,0) circle (0.5);
  \filldraw[white] (2,5) circle (0.5);
  \filldraw[white] (10,5) circle (0.5);
  \filldraw[white] (18,5) circle (0.5);
  \draw (0,0) circle (0.5);
  \draw (2,0) circle (0.5);
  \draw (4,0) circle (0.5);
  \draw (9,0) circle (0.5);
  \draw (12,0) circle (0.5);
  \draw (16,0) circle (0.5);
  \draw (18,0) circle (0.5);
  \draw (20,0) circle (0.5);
  \draw (2,5) circle (0.5);
  \draw (10,5) circle (0.5);
  \draw (18,5) circle (0.5);    
  \draw (0,0) node{0};
  \draw (2,0) node{1};
  \draw (4,0) node{2};
  \draw (9,0) node{34};
  \draw (12,0) node{5};
  \draw (16,0) node{6};
  \draw (18,0) node{7};
  \draw (20,0) node{8};
  \draw (2,5) node{9};
  \draw (10,5) node{10};
  \draw (18,5) node{11};
\end{tikzpicture}  
\end{center}
\end{figure}
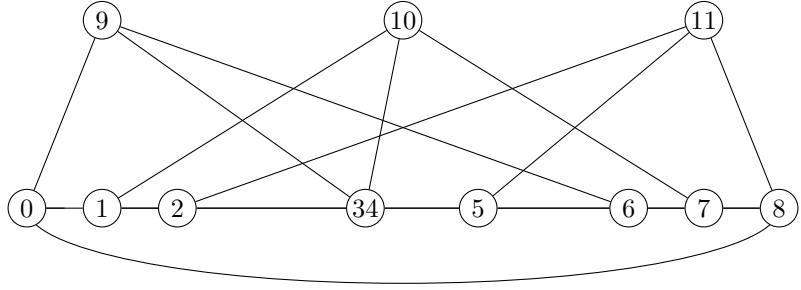

  $G$ is trivalent at every vertex except vertex 34, which is
  4-valent.

  Bayer and Eisenbud introduced a theory of graph curves in \cite{BE}. These are
  nodal curves for which each irreducible component is a
  rational curve. The graph in the name is the dual graph of the
  curve. (Note: in \cite{BE}, the definition of a graph curve specifies that the graph should be trivalent, but we will continue to call the objects we study
  graph curves even though there is one 4-valent vertex.)

The graph $G$ defines a 1-dimensional family of nodal curves  because
we can vary the cross-ratio of the four nodes on the component labeled
34. The graph $G$ has three specializations to trivalent graphs as
indicated in Figure \ref{TrivalentSpecializationsFigure}.

\begin{figure}[h]
  \begin{center}
    \caption{The graph $G$ and its trivalent specializations}
       \label{TrivalentSpecializationsFigure}    
    \begin{tikzpicture}[scale=0.35]
\begin{scope}[xshift=0cm,yshift=0cm]      
  \draw (4,0)--(0,-3); 
  \draw (4,0)--(7,-1); 
  \draw (4,0)--(5,-6); 
  \draw (0,-3)--(1,-8); 
  \draw (0,-3)--(3,-4); 
  \draw[very thick,red] (1,-8)--(4,-10.5); 
  \draw (1,-8)--(8,-7); 
  \draw[very thick,red] (4,-10.5)--(5,-6); 
  \draw[very thick,red] (4,-10.5)--(10,-9); 
  \draw[very thick,red] (4,-10.5)--(3,-4); 
  \draw (10,-9)--(11,-5); 
  \draw (10,-9)--(8,-7); 
  \draw (11,-5)--(9,-2); 
  \draw (11,-5)--(5,-6); 
  \draw (9,-2)--(7,-1); 
  \draw (9,-2)--(3,-4); 
  \draw (7,-1)--(8,-7); 
  \filldraw[white] (4,0) circle (0.5); 
  \filldraw[white] (0,-3) circle (0.5); 
  \filldraw[white] (1,-8) circle (0.5); 
  \filldraw[white] (4,-10.5) circle (0.5); 
  \filldraw[white] (10,-9) circle (0.5); 
  \filldraw[white] (11,-5) circle (0.5); 
  \filldraw[white] (9,-2) circle (0.5); 
  \filldraw[white] (7,-1) circle (0.5); 
  \filldraw[white] (5,-6) circle (0.5); 
  \filldraw[white] (3,-4) circle (0.5); 
  \filldraw[white] (8,-7) circle (0.5); 
  \draw (4,0) circle (0.5); 
  \draw (0,-3) circle (0.5); 
  \draw (1,-8) circle (0.5); 
  \draw[red] (4,-10.5) circle (0.5); 
  \draw (10,-9) circle (0.5); 
  \draw (11,-5) circle (0.5); 
  \draw (9,-2) circle (0.5); 
  \draw (7,-1) circle (0.5); 
  \draw (5,-6) circle (0.5); 
  \draw (3,-4) circle (0.5); 
  \draw (8,-7) circle (0.5); 
  \draw (4,0) node{0};
  \draw (0,-3) node{1};
  \draw (1,-8) node{2};
  \draw (4,-10.5) node[red]{34};
  \draw (10,-9) node{5};
  \draw (11,-5) node{6};
  \draw (9,-2) node{7};
  \draw (7,-1) node{8};
  \draw (5,-6) node{9};
  \draw (3,-4) node{10};
  \draw (8,-7) node{11}; 
\end{scope}

\begin{scope}[xshift=16cm,yshift=0cm]
  \draw (4,0)--(0,-3); 
  \draw (4,0)--(7,-1); 
  \draw (4,0)--(5,-6); 
  \draw (0,-3)--(1,-8); 
  \draw (0,-3)--(3,-4); 
  \draw[very thick,red] (1,-8)--(2,-11); 
  \draw (1,-8)--(8,-7); 
  \draw[very thick,red] (2,-11)--(6,-10); 
  \draw[very thick,red] (2,-11)--(5,-6); 
  \draw[very thick,red] (6,-10)--(10,-9); 
  \draw[very thick,red] (6,-10)--(3,-4); 
  \draw (10,-9)--(11,-5); 
  \draw (10,-9)--(8,-7); 
  \draw (11,-5)--(9,-2); 
  \draw (11,-5)--(5,-6); 
  \draw (9,-2)--(7,-1); 
  \draw (9,-2)--(3,-4); 
  \draw (7,-1)--(8,-7); 
  \filldraw[white] (4,0) circle (0.5); 
  \filldraw[white] (0,-3) circle (0.5); 
  \filldraw[white] (1,-8) circle (0.5); 
  \filldraw[white] (2,-11) circle (0.5); 
  \filldraw[white] (6,-10) circle (0.5); 
  \filldraw[white] (10,-9) circle (0.5); 
  \filldraw[white] (11,-5) circle (0.5); 
  \filldraw[white] (9,-2) circle (0.5); 
  \filldraw[white] (7,-1) circle (0.5); 
  \filldraw[white] (5,-6) circle (0.5); 
  \filldraw[white] (3,-4) circle (0.5); 
  \filldraw[white] (8,-7) circle (0.5); 
  \draw (4,0) circle (0.5); 
  \draw (0,-3) circle (0.5); 
  \draw (1,-8) circle (0.5); 
  \draw[red] (2,-11) circle (0.5); 
  \draw[red] (6,-10) circle (0.5); 
  \draw (10,-9) circle (0.5); 
  \draw (11,-5) circle (0.5); 
  \draw (9,-2) circle (0.5); 
  \draw (7,-1) circle (0.5); 
  \draw (5,-6) circle (0.5); 
  \draw (3,-4) circle (0.5); 
  \draw (8,-7) circle (0.5); 
  \draw (4,0) node{0};
  \draw (0,-3) node{1};
  \draw (1,-8) node{2};
  \draw (2,-11) node[red]{3};
  \draw (6,-10) node[red]{4};
  \draw (10,-9) node{5};
  \draw (11,-5) node{6};
  \draw (9,-2) node{7};
  \draw (7,-1) node{8};
  \draw (5,-6) node{9};
  \draw (3,-4) node{10};
  \draw (8,-7) node{11}; 
\end{scope}

\begin{scope}[xshift=0cm,yshift=-16cm]
  \draw (4,0)--(0,-3); 
  \draw (4,0)--(7,-1); 
  \draw (4,0)--(5,-6); 
  \draw (0,-3)--(1,-8); 
  \draw (0,-3)--(3,-4); 
  \draw[very thick,red] (1,-8)--(6,-10); 
  \draw (1,-8)--(8,-7); 
  \draw[very thick,red] (2,-11)--(6,-10); 
  \draw[very thick,red] (2,-11)--(5,-6); 
  \draw[very thick,red] (6,-10)--(10,-9); 
  \draw[very thick,red] (2,-11)--(3,-4); 
  \draw (10,-9)--(11,-5); 
  \draw (10,-9)--(8,-7); 
  \draw (11,-5)--(9,-2); 
  \draw (11,-5)--(5,-6); 
  \draw (9,-2)--(7,-1); 
  \draw (9,-2)--(3,-4); 
  \draw (7,-1)--(8,-7); 
  \filldraw[white] (4,0) circle (0.5); 
  \filldraw[white] (0,-3) circle (0.5); 
  \filldraw[white] (1,-8) circle (0.5); 
  \filldraw[white] (2,-11) circle (0.5); 
  \filldraw[white] (6,-10) circle (0.5); 
  \filldraw[white] (10,-9) circle (0.5); 
  \filldraw[white] (11,-5) circle (0.5); 
  \filldraw[white] (9,-2) circle (0.5); 
  \filldraw[white] (7,-1) circle (0.5); 
  \filldraw[white] (5,-6) circle (0.5); 
  \filldraw[white] (3,-4) circle (0.5); 
  \filldraw[white] (8,-7) circle (0.5); 
  \draw (4,0) circle (0.5); 
  \draw (0,-3) circle (0.5); 
  \draw (1,-8) circle (0.5); 
  \draw[red] (2,-11) circle (0.5); 
  \draw[red] (6,-10) circle (0.5); 
  \draw (10,-9) circle (0.5); 
  \draw (11,-5) circle (0.5); 
  \draw (9,-2) circle (0.5); 
  \draw (7,-1) circle (0.5); 
  \draw (5,-6) circle (0.5); 
  \draw (3,-4) circle (0.5); 
  \draw (8,-7) circle (0.5); 
  \draw (4,0) node{0};
  \draw (0,-3) node{1};
  \draw (1,-8) node{2};
  \draw (2,-11) node[red]{3};
  \draw (6,-10) node[red]{4};
  \draw (10,-9) node{5};
  \draw (11,-5) node{6};
  \draw (9,-2) node{7};
  \draw (7,-1) node{8};
  \draw (5,-6) node{9};
  \draw (3,-4) node{10};
  \draw (8,-7) node{11}; 
\end{scope}

\begin{scope}[xshift=16cm,yshift=-16cm]
 \draw (4,0)--(0,-3); 
  \draw (4,0)--(7,-1); 
  \draw (4,0)--(5,-6); 
  \draw (0,-3)--(1,-8); 
  \draw (0,-3)--(3,-4); 
  \draw[very thick,red] (1,-8)--(2,-11); 
  \draw (1,-8)--(8,-7); 
  \draw[very thick,red] (6,-10)--(2,-11); 
  \draw[very thick,red] (6,-10)--(5,-6); 
  \draw[very thick,red] (6,-10)--(10,-9); 
  \draw[very thick,red] (2,-11)--(3,-4); 
  \draw (10,-9)--(11,-5); 
  \draw (10,-9)--(8,-7); 
  \draw (11,-5)--(9,-2); 
  \draw (11,-5)--(5,-6); 
  \draw (9,-2)--(7,-1); 
  \draw (9,-2)--(3,-4); 
  \draw (7,-1)--(8,-7); 
  \filldraw[white] (4,0) circle (0.5); 
  \filldraw[white] (0,-3) circle (0.5); 
  \filldraw[white] (1,-8) circle (0.5); 
  \filldraw[white] (6,-10) circle (0.5); 
  \filldraw[white] (2,-11) circle (0.5); 
  \filldraw[white] (10,-9) circle (0.5); 
  \filldraw[white] (11,-5) circle (0.5); 
  \filldraw[white] (9,-2) circle (0.5); 
  \filldraw[white] (7,-1) circle (0.5); 
  \filldraw[white] (5,-6) circle (0.5); 
  \filldraw[white] (3,-4) circle (0.5); 
  \filldraw[white] (8,-7) circle (0.5); 
  \draw (4,0) circle (0.5); 
  \draw (0,-3) circle (0.5); 
  \draw (1,-8) circle (0.5); 
  \draw[red] (6,-10) circle (0.5); 
  \draw[red] (2,-11) circle (0.5); 
  \draw (10,-9) circle (0.5); 
  \draw (11,-5) circle (0.5); 
  \draw (9,-2) circle (0.5); 
  \draw (7,-1) circle (0.5); 
  \draw (5,-6) circle (0.5); 
  \draw (3,-4) circle (0.5); 
  \draw (8,-7) circle (0.5); 
  \draw (4,0) node{0};
  \draw (0,-3) node{1};
  \draw (1,-8) node{2};
  \draw (6,-10) node[red]{3};
  \draw (2,-11) node[red]{4};
  \draw (10,-9) node{5};
  \draw (11,-5) node{6};
  \draw (9,-2) node{7};
  \draw (7,-1) node{8};
  \draw (5,-6) node{9};
  \draw (3,-4) node{10};
  \draw (8,-7) node{11}; 
 \end{scope}

\draw (5,1.5) node{$G$};        
\draw (21,1.5) node{$G_0$};
\draw (5,-14.5) node{$G_{1}$};
\draw (21,-14.5) node{$G_{\infty}$};
\end{tikzpicture}
\end{center}
\end{figure}
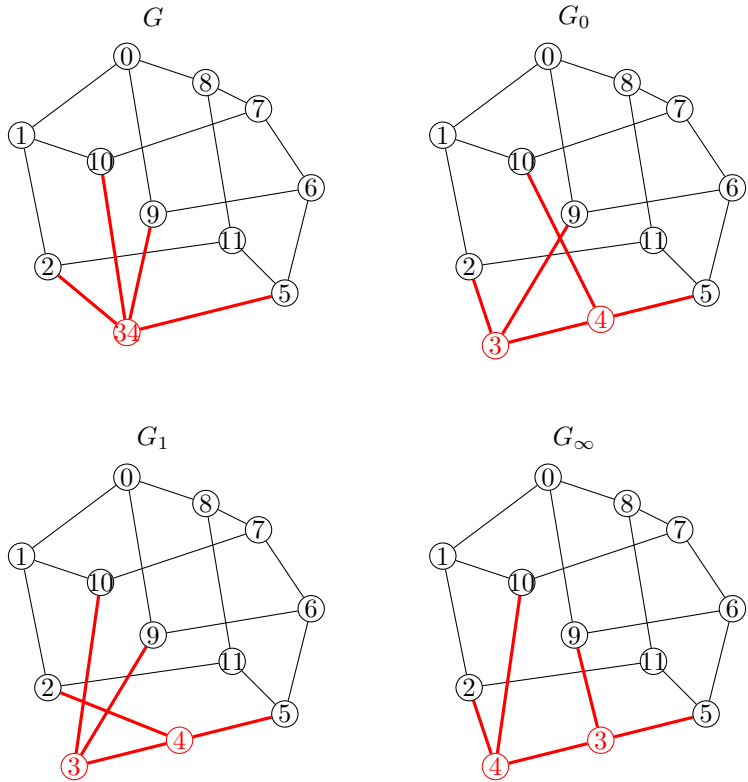

\subsection{How this family was selected}

This family was selected as follows.

We searched for genus 7 trivalent graph curves with pure
Betti tables. In \texttt{Sage}, we called a list of the genus 7
trivalent graphs; there are 85 such graphs. Next, we selected the
connected and 3-edge-connected graphs among this list, since by \cite{BE}*{Prop. 2.5},
these are the ones that give graph curves with very ample dualizing
sheaves. Next, we computed the Betti tables of these graph curves in
\texttt{Macaulay2} and found two genus 7 graph curves with pure Betti
tables. We selected the one that had the larger automorphism group for
further study; see \codelink{https://faculty.fordham.edu/dswinarski/MukaiModelOfM7/v1/5.1.htm}{5.1}. This is the graph $G_0$ in Figure \ref{K33FigureG0}. It has two types of edges: 
those that belong to the nonagon, and those that do not. We 
contracted one of the nonagon edges to obtain the graph $G$.

For every member of this family of curves, the
dualizing sheaf is very ample. Thus, each member of the family is
represented in the Hilbert scheme of canonical curves. Moreover, for a
general member of this family, and the specialization $G_0$ (but not
the specializations $G_1$
and $G_{\infty}$), the canonical ideal has a pure Betti table. This permits us to study
degenerations in the parameter space of Mukai's model as the curve
acquires extra syzygies.

The relevant combinatorial features of the graphs are that $G_1$ and
$G_\infty$ each contain 4-cycles, whereas in $G_0$, the shortest cycles
have length 5. Bayer and Eisenbud describe in \cite{BE}*{Section 5} how to
construct line bundles that lower the Clifford index and add to the
Betti table starting from cycles that
are sufficiently small relative to the genus of the graph.

\subsection{Canonical equations of these graph curves}

To produce equations for this family, we begin with the specialization
$G_{0}$. See Figure \ref{K33FigureG0}.  
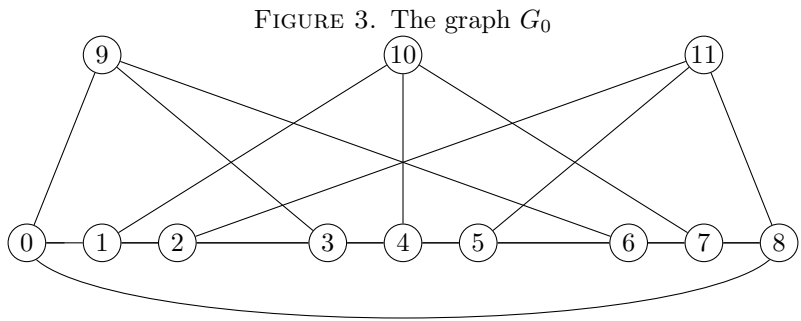
\begin{figure}[h]
  \begin{center}
    \caption{The graph $G_0$}
        \label{K33FigureG0}    
\begin{tikzpicture}[scale=0.5]
  \draw (0,0)--(1,0); 
  \draw (0,0) arc (180:360:10 and 2);
  \draw (0,0)--(20,0); 
  \draw (0,0)--(2,5); 
  \draw (2,0)--(4,0); 
  \draw (2,0)--(10,5); 
  \draw (4,0) -- (8,0); 
  \draw (4,0)--(18,5); 
  \draw (8,0)--(10,0); 
  \draw (8,0)--(2,5); 
  \draw (10,0)--(12,0); 
  \draw (10,0)--(10,5); 
  \draw (12,0) -- (16,0); 
  \draw (12,0)--(18,5); 
  \draw (16,0)--(18,0); 
  \draw (16,0)--(2,5); 
  \draw (18,0)--(20,0); 
  \draw (18,0)--(10,5); 
  \draw (20,0)--(18,5); 
  \filldraw[white] (0,0) circle (0.5);
  \filldraw[white] (2,0) circle (0.5);
  \filldraw[white] (4,0) circle (0.5);
  \filldraw[white] (8,0) circle (0.5);
  \filldraw[white] (10,0) circle (0.5);
  \filldraw[white] (12,0) circle (0.5);
  \filldraw[white] (16,0) circle (0.5);
  \filldraw[white] (18,0) circle (0.5);
  \filldraw[white] (20,0) circle (0.5);
  \filldraw[white] (2,5) circle (0.5);
  \filldraw[white] (10,5) circle (0.5);
  \filldraw[white] (18,5) circle (0.5);
  \draw (0,0) circle (0.5);
  \draw (2,0) circle (0.5);
  \draw (4,0) circle (0.5);
  \draw (8,0) circle (0.5);
  \draw (10,0) circle (0.5);
  \draw (12,0) circle (0.5);
  \draw (16,0) circle (0.5);
  \draw (18,0) circle (0.5);
  \draw (20,0) circle (0.5);
  \draw (2,5) circle (0.5);
  \draw (10,5) circle (0.5);
  \draw (18,5) circle (0.5);    
  \draw (0,0) node{0};
  \draw (2,0) node{1};
  \draw (4,0) node{2};
  \draw (8,0) node{3};
  \draw (10,0) node{4};
  \draw (12,0) node{5};
  \draw (16,0) node{6};
  \draw (18,0) node{7};
  \draw (20,0) node{8};
  \draw (2,5) node{9};
  \draw (10,5) node{10};
  \draw (18,5) node{11};
\end{tikzpicture}  
\end{center}
\end{figure}

Let $\CNod{0}$ be the graph curve associated to the graph $G_0$. Since $G_0$ is
3-edge-connected, by \cite{BE}*{Prop.~2.5}, $\omega_{\CNod{0}}$ is very
ample. We can use \cite{BE}*{Prop.~3.1} to write the canonical ideal
of $\CNod{0}$; see \codelink{https://faculty.fordham.edu/dswinarski/MukaiModelOfM7/v1/5.2.htm}{5.2}. Let $y_0,\ldots,y_6$ represent the basis of $H^1(G_0)$
corresponding to the 5-cycles 0-1-2-3-9-0, 1-2-3-4-10-1, 2-3-4-5-11-2,
3-4-5-6-9-3, 4-5-6-7-10-4, 5-6-7-8-11-5, and 6-7-8-0-9-6. Then the
canonical ideal of $\CNod{0}$ in these variables is given by the following 10
quadrics. There are 5 monomials and 5 polynomials. 
\[
\begin{array}{c}
I(\CNod{0})  = \langle y_0 y_4, y_0 y_5, y_1 y_5, y_1 y_6, y_2 y_6,  
y_0 y_2-y_1 y_2+y_2 y_3-y_3 y_4+y_4 y_5-y_4 y_6,   \\
y_0 y_3-y_2 y_3+y_3^2-y_4 y_5+y_3 y_6+y_4 y_6, 
y_1 y_3-y_2 y_3+y_3 y_4-y_4 y_5+y_4 y_6, \\
y_2 y_4-y_3 y_4+y_4 y_5-y_4 y_6, 
  y_3 y_5-y_4 y_5+y_4 y_6 \rangle
\end{array}
\]

We can compute a primary decomposition of the ideal shown above to
obtain the ideal of each irreducible component of $\CNod{0}$. This
yields Table \ref{tab:Components of C0}.
\begin{table}
  \caption{$\CNod{0}^{v}$ for $v \in G_0$}
  \label{tab:Components of C0}
\begin{center}
  \begin{tabular}{cl}
Vertex & Ideal of component in $\CNod{0}$\\    
0 & $\langle y_5,y_4,y_3,y_2,y_1\rangle$ \\
1 & $\langle y_6,y_5,y_4,y_3,y_2\rangle$ \\
2 & $\langle y_6,y_5,y_4,y_3,y_0-y_1\rangle$ \\
3 & $\langle y_6,y_5,y_4,y_1-y_2,y_0-y_2+y_3\rangle$ \\
4 & $\langle y_6,y_5,y_2-y_3,y_1-y_3+y_4,y_0\rangle$ \\
5 & $\langle y_6,y_3-y_4,y_2-y_4+y_5,y_1,y_0\rangle$ \\
6 & $\langle y_4-y_5,y_3-y_5+y_6,y_2,y_1,y_0\rangle$ \\
7 & $\langle y_5-y_6,y_3,y_2,y_1,y_0\rangle$ \\
8 & $\langle y_4,y_3,y_2,y_1,y_0\rangle$ \\
9 & $\langle y_5,y_4,y_2,y_1,y_0+y_3+y_6\rangle$ \\
10 & $\langle y_6,y_5,y_3,y_2,y_0\rangle$ \\
11 & $\langle y_6,y_4,y_3,y_1,y_0\rangle$  
\end{tabular}                                    
\end{center}
\end{table}

Next, we find equations for the other members of this family by
replacing the components 3 and 4 by a quadric; see \codelink{https://faculty.fordham.edu/dswinarski/MukaiModelOfM7/v1/5.3.htm}{5.3}.
The union of components 3 and 4 in $\CNod{0}$ is contained in the plane
$\langle y_6,y_5,y_1-y_2+y_4,y_0-y_2+y_3\rangle$. The nodes corresponding to the
edges 2-3, 3-9, 4-5, and 4-10 occur at $[1:1:1:0:0:0:0]$,
$[-1:0:0:1:0:0:0]$, $[0:0:1:1:1:0:0]$ and $[0:-1:0:0:1:0:0]$. For all
$t = [t_0:t_1]$, the quadric $t_0 y_2 y_3-t_1 y_2 y_4+(-t_0+t_1) y_3 y_4$
in this plane passes through these four points. When $t_0=0$, the
quadric factors as $(y_2-y_3)y_4$, which corresponds to the graph
$G_0$. When $t_0=t_1$, the quadric factors as  $y_2(y_3-y_4)$, which corresponds
to the graph $G_{1}$.  When $t_1=0$, the quadric factors as $y_3(y_2-y_4)$, which
corresponds to the graph $G_{\infty}$. 

Now, for a general $t$, we intersect the ideals for components $0,1,2,5,6,7,8,9,10,11$ with the ideal
\[
\langle t_0 y_2 y_3-t_1 y_2 y_4+(-t_0+t_1) y_3 y_4, y_6,y_5,y_1-y_2+y_4,y_0-y_2+y_3\rangle
\]
defining the component 34 to obtain an ideal $I_t$ generated by the following ten quadrics.
\begin{align*}
  f_0 &= y_2 y_6\\
  f_1 &= y_1 y_6\\
  f_2 &= y_3 y_5-y_4 y_5+y_4 y_6\\
  f_3 &= y_1 y_5\\
  f_4 &= y_0 y_5 \\
  f_5 &= y_0 y_4-y_2 y_4+y_3 y_4-y_4 y_5+y_4 y_6\\
  f_6 &= t_0 y_2 y_3-t_1 y_2 y_4+(-t_0+t_1) y_3 y_4+(t_0-t_1) y_4 y_5+(-t_0+t_1) y_4 y_6\\
  f_7&= y_1 y_3-y_2 y_3+y_3 y_4-y_4 y_5+y_4 y_6\\
  f_8 &= y_0 y_3-y_2 y_3+y_3^2-y_4 y_5+y_3 y_6+y_4 y_6 \\
  f_9 &= y_0 y_2-y_1 y_2+y_2 y_3-y_2 y_4 
\end{align*}

We change to the following basis of $I_t$ so that $\sum_{i=0}^{4} g_i g_{i+5} = 0$; see \codelink{https://faculty.fordham.edu/dswinarski/MukaiModelOfM7/v1/5.4.htm}{5.4}.
\begin{align*}
g_0 &= -t_0 y_1 y_3+t_1 y_0 y_4\\
g_1 &= -(t_0-t_1) y_0 y_3-t_1 y_2 y_3-(t_0-t_1) y_3^2+t_1 y_2 y_4+(t_0-t_1) y_3y_4-(t_0-t_1) y_3 y_6 \\
g_2 &= (t_0-t_1) y_1 y_3+t_1 y_2 y_3-t_1 y_2 y_4 \\
g_3 &= t_0 y_2 y_3-t_1 y_2 y_4-(t_0-t_1) y_3 y_4+(t_0-t_1) y_3 y_5 \\
g_4 &= -t_1 y_0 y_2+t_1 y_1 y_2+-t_1 y_2 y_3+t_1 y_2 y_4\\
g_5 &= y_2 y_6 \\
g_6 &= y_1 y_5 \\
g_7 &= y_0 y_5+y_3 y_5-y_4 y_5+y_4 y_6 \\
g_8 &= y_1 y_6 \\
g_9 &= y_3 y_5-y_4 y_5+y_4 y_6
\end{align*}

For $t \in (\mathbb{P}^1 \setminus \{1,\infty\})$, the Betti table of
$I_t$ is pure. When $t \in \{1,\infty\}$, the Betti table is 
\begin{center}
\begin{verbatim}
                           0  1  2  3  4 5
                    total: 1 10 19 19 10 1
                        0: 1  .  .  .  . .
                        1: . 10 16  3  . .
                        2: .  .  3 16 10 .
                        3: .  .  .  .  . 1
\end{verbatim}
\end{center}
By \cite{Schreyer}, this is the Betti table of a tetragonal curve.

\newpage

\begin{table}
  \caption{$X_{t}^{v}$ for $v \in G$}
  \label{tab:components of Xt}
\begin{center}
  \begin{tabular}{lll}
Component & Ideal \\
0 & $\langle x_{1234}, x_{1235}, x_{1245}, x_{2345}, x_{45}, x_{35}, x_{25}, x_{24}, x_{23}, x_{15}, x_{14}, x_{13}, x_{12}, x_{0}\rangle$ \\
1 & $\langle x_{1234}, x_{1235}, x_{1245}, x_{1345}, x_{2345}, x_{45}, x_{35}, x_{25}, x_{23}, x_{15}, x_{14}, x_{13}, x_{12}, x_{0}\rangle$ \\
2 & $\langle x_{1234}, x_{1235}, x_{1245}, x_{1345}, x_{2345}, x_{45}, x_{35}, x_{25}, x_{24}-x_{34}, x_{23}, x_{15}, x_{14}, x_{12}-x_{13}, x_{0}\rangle$ \\ 
34 & $\langle x_{1234}, x_{1235}, x_{1245}, x_{1345}, x_{2345},
     x_{24}-x_{34}, x_{23}-x_{25}, x_{15}+x_{25}-x_{35}-x_{45},
     x_{14},$\\
          & $x_{13}+x_{25}+x_{34}, x_{12}+x_{25}+x_{34}, x_{25} x_{34}-x_{34} x_{35}+x_{25} x_{45}, t_1 x_{25}-t_1 x_{35}-t_1 x_{45}+x_{0}, $\\
          & $t_0 x_{25}-t_1 x_{35}-t_0 x_{45}, t_1 x_{34} x_{45}+t_1 x_{35} x_{45}+t_1 x_{45}^2-x_{0} x_{34}-x_{0} x_{45}, $\\
  &  $t_0 x_{34} x_{35}-t_1 x_{34} x_{35}-t_0 x_{34} x_{45}-t_1 x_{35} x_{45}-t_0 x_{45}^2, t_0 t_1 x_{35}-t_1^2 x_{35}-t_0 x_{0}\rangle$ \\
    5 & $\langle x_{1234}, x_{1245}, x_{1345}, x_{2345}, x_{45}, x_{34}, x_{24}, x_{23}-x_{25}, x_{15}+x_{25}-x_{35}, x_{14}, x_{12}-x_{13}, $\\
          & $t_1 x_{1235}-x_{13}-x_{25}, t_1 x_{25}-t_1 x_{35}+x_{0}, t_0 x_{25}-t_1 x_{35}, $\\
    & $t_0 x_{13} x_{35}-t_1 x_{13} x_{35}-t_0 x_{0} x_{1235}+x_{0} x_{35}, t_0 t_1 x_{35}-t_1^2 x_{35}-t_0 x_{0}\rangle$ \\ 
    6 & $\langle x_{1234}, x_{1245}, x_{2345}, x_{34}+x_{45}, x_{24}, x_{23}-x_{25}, x_{15}+x_{25}-x_{35}-x_{45}, x_{14}, x_{13}, x_{12}, $\\
          & $t_1 x_{1235}-x_{25}, t_1 x_{1345}+x_{45}, t_0 x_{1345}+t_0 x_{1235}-x_{35}, t_1 x_{25}-t_1 x_{35}-t_1 x_{45}+x_{0}, $\\
    & $t_0 x_{25}-t_1 x_{35}-t_0 x_{45}, t_0 t_1 x_{35}-t_1^2 x_{35}-t_0 x_{0}\rangle$ \\ 
7 & $\langle x_{1234}, x_{1245}, x_{1345}+x_{1235}, x_{2345}, x_{35}, x_{34}+x_{45}, x_{25}-x_{45}, x_{24}, x_{23}-x_{45}, x_{15}, x_{14}, x_{13}, x_{12}, x_{0}\rangle$ \\ 
8 & $\langle x_{1234}, x_{1245}, x_{2345}, x_{45}, x_{35}, x_{34}, x_{25}, x_{24}, x_{23}, x_{15}, x_{14}, x_{13}, x_{12}, x_{0}\rangle$ \\ 
    9 & $\langle x_{1234}, x_{1235}, x_{1245}, x_{2345}, x_{25}, x_{24}, x_{23}, x_{15}-x_{35}-x_{45}, x_{14}, x_{13}, x_{12}, t_1 x_{1345}-x_{34}, $\\
    & $t_0 x_{45}-t_1 x_{45}+x_{0}, x_{34} x_{35}+x_{34} x_{45}-x_{0} x_{1345}, t_1 x_{35}+t_1 x_{45}-x_{0}\rangle$ \\ 
10 & $\langle x_{1234}, x_{1235}, x_{1245}, x_{1345}, x_{2345}, x_{35}, x_{34}+x_{45}, x_{25}-x_{45}, x_{23}-x_{45}, x_{15}, x_{14}, x_{13}, x_{12}, x_{0}\rangle$ \\ 
11 & $\langle x_{1234}, x_{1245}, x_{1345}, x_{2345}, x_{45}, x_{35}, x_{34}, x_{25}, x_{24}, x_{23}, x_{15}, x_{14}, x_{12}-x_{13}, x_{0}\rangle$
\end{tabular}                                    
\end{center}
\end{table}

\begin{table}
  \caption{Nodes of $X_t$ for generic $t$}
  \label{tab:Nodes of Xt}
\begin{center}
\begin{tabular}{ll}
0 -- 1 & $[0 : 0 : 0 : 0 : 0 : 0 : 0 : 0 : 1 : 0 : 0 : 0 : 0 : 0 : 0 : 0]$ \\
1 -- 2 & $[0 : 0 : 0 : 0 : 0 : 0 : 1 : 0 : 1 : 0 : 0 : 0 : 0 : 0 : 0 : 0]$ \\
2-- 34 & $[0 : -1 : -1 : 0 : 0 : 0 : 1 : 0 : 1 : 0 : 0 : 0 : 0 : 0 : 0 : 0]$ \\
34 -- 5 & $[t_0 t_1-t_1^2 : -t_1 : -t_1 : 0 : t_0-t_1 : t_1 : 0 : t_1 : 0 : t_0 : 0 : 0 : 0 : 0 : 0 : 0]$ \\
5 -- 6 & $[t_0t_1-t_1^2 : 0 : 0 : 0 : t_0-t_1 : t_1 : 0 : t_1 : 0 : t_0 : 0 : 0 : 1 : 0 : 0 : 0]$ \\
6 -- 7 & $[0 : 0 : 0 : 0 : 0 : t_1 : 0 : t_1 : -t_1 : 0 : t_1 : 0 : 1 : 0 : -1 : 0]$ \\
0 -- 8 & $[0 : 0 : 0 : 0 : 0 : 0 : 0 : 0 : 0 : 0 : 0 : 0 : 0 : 0 : 1 : 0]$ \\
7 -- 8 & $[0 : 0 : 0 : 0 : 0 : 0 : 0 : 0 : 0 : 0 : 0 : 0 : 1 : 0 : -1 : 0]$ \\
0 -- 9 & $[0 : 0 : 0 : 0 : 0 : 0 : 0 : 0 : t_1 : 0 : 0 : 0 : 0 : 0 : 1 : 0]$ \\
34 -- 9 & $[-t_0 t_1+t_1^2 : 0 : 0 : 0 : -t_0+t_1 : 0 : 0 : 0 : 0 : -t_0 : t_1 : 0 : 0 : 0 : 0 : 0]$ \\
6 -- 9 & $[t_0 t_1-t_1^2 : 0 : 0 : 0 : t_0-t_1 : 0 : 0 : 0 : t_1 : t_0 : -t_1 : 0 : 0 : 0 : 1 : 0]$ \\
1 -- 10 & $[0 : 0 : 0 : 0 : 0 : 0 : 1 : 0 : 0 : 0 : 0 : 0 : 0 : 0 : 0 : 0]$ \\
34 -- 10 & $[0 : 0 : 0 : 0 : 0 : 1 : -1 : 1 : -1 : 0 : 1 : 0 : 0 : 0 : 0 : 0]$ \\
7 -- 10 & $[0 : 0 : 0 : 0 : 0 : 1 : 0 : 1 : -1 : 0 : 1 : 0 : 0 : 0 : 0 : 0]$ \\
2 -- 11 & $[0 : 1 : 1 : 0 : 0 : 0 : 0 : 0 : 0 : 0 : 0 : 0 : 0 : 0 : 0 : 0]$ \\
5 -- 11 & $[0 : t_1 : t_1 : 0 : 0 : 0 : 0 : 0 : 0 : 0 : 0 : 0 : 1 : 0 : 0 : 0]$ \\
8 -- 11 & $[0 : 0 : 0 : 0 : 0 : 0 : 0 : 0 : 0 : 0 : 0 : 0 : 1 : 0 : 0 : 0]$
\end{tabular}
\end{center}
\end{table}

\newpage

\subsection{Spinor embeddings of each component of $\CNod{t}$}

Next, for $t \not\in \{0,1,\infty\}$, we embed each irreducible component of $\CNod{t}$ in
$\mathbb{P}(S^{+})$ and define
\begin{align*}
    X_{t}^{v} &:= \rho(\CNod{t}^v) \text{ for each } v \in G \\
    X_t &:= \bigcup_{v \in G} X_t^{v}
\end{align*}

We compute each component $X_{t}^{v}$ as
follows. First, we parametrize $\CNod{t}^v$. Next, we compute the
spinor associated to $W_{p}^{\perp}$, where $p$ is the general point
given by the parametrization of $\CNod{t}^v$. This gives us a parametrization of the line $X_{t}^{v}$  in $\mathbb{P}(S^{+})$. We then eliminate parameters to obtain
the ideal of $X_{t}^{v}$ in $\mathbb{P}(S^{+})$. This yields Table
\ref{tab:components of Xt}; see
\codelink{https://faculty.fordham.edu/dswinarski/MukaiModelOfM7/v1/5.5.htm}{5.5}. (Note:
the generators shown are not necessarily a Gr\"{o}bner basis in each
case.)

For $t \not\in \{0,1,\infty\}$, these components intersect at the
nodes listed in Table \ref{tab:Nodes of Xt}.

Next, we compute the ideal of $X_t = \bigcup_{v \in G} X_t^{v}$; see \codelink{https://faculty.fordham.edu/dswinarski/MukaiModelOfM7/v1/5.6.htm}{5.6}.  (Here we show a minimal set of generators; a Gr\"{o}bner basis is used for the limit computations.)

\begin{equation}
  \begin{array}{c}
   I( X_t) = \langle  x_{2345}, x_{1245}, x_{1234}, t_0 x_{25}-t_1 x_{35}-t_0 x_{45}, x_{23}-x_{25}, x_{15}+x_{25}-x_{35}-x_{45}, x_{14}, \\
    x_{12}-x_{13}, x_0+t_1 x_{25}-t_1 x_{35}-t_1 x_{45}, x_{24}
    x_{1345}, x_{13} x_{1345}, x_{45} x_{1235} + x_{25} x_{1345},    x_{34} x_{1235}-x_{25} x_{1345}, \\
    x_{24} x_{1345}, x_{13} x_{45}+x_{25} x_{45}+x_{34} x_{45}+t_1 x_{25} x_{1345}-t_1 x_{45} x_{1345}, \\
    x_{24} x_{35}-x_{34} x_{35}+t_1 x_{35} x_{1345},\\
    x_{13} x_{35}+x_{25} x_{35}+x_{34} x_{35}-t_1 x_{35} x_{1235}-t_1
    x_{35} x_{1345}, \\ 
    x_{25} x_{34}-x_{34} x_{35}+x_{25} x_{45}+t_1 x_{35} x_{1345}, \\
    x_{24} x_{25}-x_{34} x_{35}-x_{24} x_{45}+x_{25} x_{45}+x_{34}
    x_{45}+t_1 x_{25} x_{1345}+t_1 x_{35}x_{1345} -t_1 x_{45} x_{1345}, \\
    x_{13} x_{25}+x_{25}^2+x_{34} x_{35}-x_{25} x_{45}-t_1 x_{25}
    x_{1235}-t_1 x_{25} x_{1345}- t_1 x_{35} x_{1345},\\
    x_{13} x_{24}-x_{13} x_{34}, x_{35}^2 x_{1235}x_{1345} -t_0 x_{35} x_{1235}^2 x_{1345} -t_0 x_{35} x_{1235} x_{1345}^2 \rangle
  \end{array}
\end{equation}

We define $X_0$, $X_1$, and $X_{\infty}$ as the flat limits of the
family $X_t$ as $t$ approaches $0$, $1$, and $\infty$.

The ideal $I(X_t)$ contains nine generators in degree 1. We use them
to define families $\PNod{t} \subset \Gr(7,16)$ and $Y_t \subset \mathbb{P}^{15}$.

\begin{align}
\PNod{t} &:= \langle  x_{2345}, x_{1245}, x_{1234}, t_0 x_{25}-t_1x_{35}-t_0 x_{45}, x_{23}-x_{25},x_{15}+x_{25}-x_{35}-x_{45}, x_{14}, \\
\nonumber        & \qquad \qquad      x_{12}-x_{13}, x_0+t_1 x_{25}-t_1 x_{35}-t_1 x_{45} \rangle \\
Y_t &:= \PNod{t} \cap \OG
\end{align}

We establish the following propositions via explicit calculations in \texttt{Macaulay2}.

\begin{proposition} \label{prop:PNodt}
For $t \not\in \{1,\infty\}$, the map  $\rho:
\CNod{t}^{\operatorname{sm}} \rightarrow \mathbb{P}^{15}$ extends to an embedding, and
  \[
  \rho(\CNod{t}) = \PNod{t}  \cap \OG
\]
where
\begin{equation} \label{MNodt}
\MNod{t}= \left[\begin{array}{rrrrrrrrrrrrrrrr}
0 & 1 & 1 & 0 & 0 & 0 & 0 & 0 & 0 & 0 & 0 & 0 & 0 & 0 & 0 & 0\\
0 & 0 & 0 & 0 & 0 & 0 & 1 & 0 & 0 & 0 & 0 & 0 & 0 & 0 & 0 & 0\\
0 & 0 & 0 & 0 & 0 & 0 & 0 & 0 & 1 & 0 & 0 & 0 & 0 & 0 & 0 & 0\\
t_0 t_1-t_1^2 & 0 & 0 & 0 & t_0-t_1 & t_1 & 0 & t_1 & 0 & t_0 & 0 & 0 & 0 & 0 & 0 & 0\\
0 & 0 & 0 & 0 & 0 & 1 & 0 & 1 & 0 & 0 & 1 & 0 & 0 & 0 & 0 & 0\\
0 & 0 & 0 & 0 & 0 & 0 & 0 & 0 & 0 & 0 & 0 & 0 & 1 & 0 & 0 & 0\\
0 & 0 & 0 & 0 & 0 & 0 & 0 & 0 & 0 & 0 & 0 & 0 & 0 & 0 & 1 & 0
\end{array} \right]
\end{equation}
and $\PNod{t} = \mathbb{P}(\operatorname{RowSpace} \MNod{t})$
\end{proposition}  
\begin{proof}
See
\codelink{https://faculty.fordham.edu/dswinarski/MukaiModelOfM7/v1/5.7.htm}{5.7}
for the case $t \not\in
\{0,1,\infty\}$, and
\codelink{https://faculty.fordham.edu/dswinarski/MukaiModelOfM7/v1/5.8.htm}{5.8}
for the case $t = 0$.
\end{proof}

\subsection{The limits of this family as $t \rightarrow 1,\infty$}

The next two propositions describe the limits of the families $X_t$ and
$Y_t$ as $t$ approaches $1$ or $\infty$.

First, we describe the flat limits of $  X_t  $ in $\mathbb{P}(S^{+})$
as $t$ approaches $1$ or $\infty$, that is, degenerations of this
family of curves in the Hilbert scheme $\Hilb(\Pro^{15},12t-6)$. 

\begin{proposition}
  \begin{enumerate}
  \item $X_1$ is the union of the limits of the
    irreducible components in $X_t$ as $t \rightarrow 1$. It is a  graph curve whose dual graph is $G_1$. However, $\PNod{1} \cap \OG \neq X_1$.
  \item $X_{\infty}$ is the union of the limits of the
    irreducible components in $X_t$ as $t \rightarrow \infty$. It is a reducible curve that has nodes and spatial triple points as its singularities. Furthermore, $\PNod{\infty}  \cap \OG \neq X_{\infty}$.
\end{enumerate}
\end{proposition}
\begin{proof} See
  \codelink{https://faculty.fordham.edu/dswinarski/MukaiModelOfM7/v1/5.9.htm}{5.9}
  for the case $t=1$
  and
\codelink{https://faculty.fordham.edu/dswinarski/MukaiModelOfM7/v1/5.10.htm}{5.10}
for the case $t=\infty$.
\end{proof}

Here are a few more details about the curve $X_{\infty}$. In
this limit, the component defined by vertex 34 does not
split into two lines, at least over $\mathbb{Q}$. Furthermore, some of the nodes in $X_t$ collide as $t \rightarrow \infty$. Specifically, node 6--7 approaches node 7--8; node 5--11 approaches node 8--11; and node 0--9 approaches node 0--8. By computing the tangent cones at these points, we can check that these singularities are spatial triple points.

Next, we describe the limits of the family $Y_t = \PNod{t} \cap \OG$ in $\Gr(7,16)$ as $t$ approaches $1$ or $\infty$.

\begin{proposition} \label{prop:Grassmannian limits} \mbox{} \\
  \begin{enumerate}
\item $Y_1$ is a union of five irreducible components, each of dimension 2.
\begin{itemize}
\item Lines 0 and 9 in the flat limit $X_1$ are replaced by their span 
\item Lines 1, 10, and 3 in the flat limit $X_1$ are replaced by the
  scroll connecting a point $p$ on line 1 to its image on line 3 under
  the isomorphism mapping nodes 0-1, 1-2, and 1-10 to 3-9, 3-4, and 3-10.
\item Lines 2 and 4 in the flat limit $X_1$ are replaced by their span
\item Lines 5 and 11 in the flat limit $X_1$ are replaced by their span
\item Lines 6, 7, 8 in the flat limit $X_1$ are replaced by the
  scroll connecting a point $p$ on line 6 to its image on line 8 under
  the isomorphism mapping nodes 5-6, 6-7, and 6-9 to 8-11, 7-8, and 0-8.
\end{itemize}

\item $Y_{\infty} $ is a union of eight irreducible components. 
\begin{itemize}
\item Lines 0, 1, 2, 7, 10, 11 in the flat limit $X_\infty$ appear as
  irreducible components of $Y_\infty$. 
\item Component 34 (an irreducible quadric) in the flat limit
  $X_\infty$ is an irreducible component of $Y_\infty$. 
\item Lines 5, 6, 8, 9 in the flat limit $X_\infty$ are replaced by their span, a $\mathbb{P}^2$
\end{itemize}

\end{enumerate}
\end{proposition}      
\begin{proof} See
  \codelink{https://faculty.fordham.edu/dswinarski/MukaiModelOfM7/v1/5.11.htm}{5.11}
  for the case $t=1$ and
\codelink{https://faculty.fordham.edu/dswinarski/MukaiModelOfM7/v1/5.12.htm}{5.12}
for the case $t=\infty$.
\end{proof}


\subsection{GIT instability for the limits as $t \rightarrow 1,\infty$}

In this section we discuss GIT semistability/instability for the
family $[\PNod{t}]$ with respect to the maximal
torus $T \subset\Spin(10)$ given by the lifts of the diagonal maximal
torus in $\SO(10)$. (Recall: we are working with the quadratic form
$\sum q_iq_{i+n}$, so there is a maximal torus consisting of diagonal matrices.)

\begin{proposition}  \label{prop:GIT instability}
$[\PNod{t}]$ is $T$-semistable with respect to
the lift of the diagonal maximal torus $T$ in $\SO(10)$ if and only $t
\not\in \{1,\infty\}$. 
\end{proposition}

\begin{proof}GIT semistability with respect to a torus $T$ can be characterized using
state polytopes.

When $t \not\in \{0,1,\infty\}$, the state of $\PNod{t}$ has 21
points, and the state polytope has 20 vertices. The trivial
character $\chi_0$ is contained in the interior of the state polytope,
so, for general $t$, $[\PNod{t}]$ is $T$-semistable.

When $t=0$, the state of $\PNod{0}$ has 16 points, and they are all
vertices of the state polytope. The trivial
character $\chi_0$ is contained in the interior of the state polytope,
so $[\PNod{0}]$ is also $T$-semistable.

When $t=1$, the state of $\PNod{1}$ has 9 points, and the state
polytope has 8 vertices. The trivial
character $\chi_0$ is not contained in the state
polytope, so this point is $T$-unstable. We compute the proximum and
find that the worst 1-parameter subgroup is in
the direction $(-2,1,1,1,1)$.

When $t=\infty$, the state of $\PNod{\infty}$ has 12 points, and they
are all vertices of the state polytope. The trivial
character $\chi_0$ is not contained in the state
polytope, so this point is also $T$-unstable. We compute the proximum
and find that the worst 1-parameter subgroup is in
the direction $(1,0,1,0,1)$.

See \codelink{https://faculty.fordham.edu/dswinarski/MukaiModelOfM7/v1/5.13.htm}{5.13}.
\end{proof}

For any maximal torus $T \subset G$, $T$-instability implies $G$-instability. But, in general,
$T$-semistability for one maximal torus gives us little information
about $G$-semistability, in the following sense. In \cite{HP}, Hyeon and Park show that in
any GIT quotient problem of semisimple group representations, every
point is semistable with respect to a  general maximal torus. In
Section \ref{sec:constructing invariants}, with a great deal more
effort, we will study $G$-semistability for $[\PNod{t}]$ with $t \not\in \{1,\infty\}$.

\section{Constructing a $\Spin(10)$-invariant polynomial for $\mywedge^7
S^{+}$} \label{sec:constructing invariants}

 Let $S^{+}$ be the half-spin representation of $\Spin(10)$. Mukai's model of $\M_7$ is the quotient $\Gr(7,16) \git \Spin(10)$. By
definition, this GIT quotient is $\operatorname{Proj}( \oplus_d (\Sym^d \mywedge^7
S^{+})^{\Spin(10)})$. In this section, we construct a $\Spin(10)$ invariant polynomial.

We begin with an approach for computing $G$ invariants in a fixed
degree. Sturmfels calls this the \emph{Lie algebra method} in
\cite{Sturmfels}*{Section 4.5}, and it is
also discussed in Derksen and Kemper's book in \cite{DK}*{Section~4.5}. This approach uses the Casimir operator on $\mathfrak{g}$. 

\begin{definition}
Let $\delta_1,\ldots,\delta_m$ be a basis of $\mathfrak{g}$, and let $\gamma_1, \ldots, \gamma_m$ be the dual basis of $\mathfrak{g}$ with
respect to the Killing form $\kappa$. The Casimir operator $c$ is defined as 
\[ 
c = \sum \delta_i * \gamma_i
\]
\end{definition}

One key property of $c$ is the following:
  \begin{proposition} If $V(\lambda)$ is an irreducible representation with highest weight
  $\lambda$, then $c$ acts as multiplication by the scalar $(\lambda,
  \lambda+2\rho)$.  (Here (,) represents the Killing form, and $\rho$ is
  half the sum of the positive roots.) 
\end{proposition}

See for instance \cite{FH}*{(25.14)}. This suggests the following strategy for computing invariants.

\begin{proposition} \label{prop:invariants are ker c}
$v$ is invariant under $G$ if and only if $v \in \ker(c)$.
\end{proposition}

It also suggests an iterative procedure for computing invariants.
\begin{proposition} \label{prop:iterative strategy}
Let $V = \bigoplus_{\lambda \in S} V_{\lambda}^{m_{\lambda}}$ be the
irreducible decomposition of $V$. Let $S' = \{ (\lambda,\lambda+2\rho) :
\lambda \in S, \lambda \neq 0\}$. Then the operator $\prod_{k \in
  S'} (c- k)$ projects $V$ to $V^{G}$.
\end{proposition}
\begin{proof}
This is \cite{DK}*{Prop.~4.5.18}, plus the observation that we
can compute the spectrum of the Casimir operator $c$ on $V$ once we
know the irreducible decomposition of $V$.
\end{proof}

Proposition \ref{prop:invariants are ker c} gives a straightforward
algorithm for finding the invariant
polynomials in a fixed degree: compute the action of $c$, and then
compute its kernel. However, $\dim V$ is so large for  the representation we want to study that computing $\ker c$
in a naive way will not work. We have $\dim \mywedge^7 S^{+} = \binom{16}{7} = 11,440.$
A character calculation shows that the lowest degree invariants are in
degree 4; see
\codelink{https://faculty.fordham.edu/dswinarski/MukaiModelOfM7/v1/6.1.htm}{6.1}.
We have  
  \[
    \dim \Sym^4 \mywedge^7 S^{+} = \binom{11440+4-1}{4} = 714,036,824,189,260. 
  \]

A standard approach to reduce the dimensions of the spaces appearing in the
calculation is to restrict to the $T$- and $W$-invariant subspace,
where $T$ is a maximal torus and $W$ is
the Weyl group. However, this is still too large; we have $\dim
(\Sym^4 \mywedge^7 S^{+})^{T} = 359,317,176,120$, which implies that
the $T$-and $W$-invariant subspace will have dimension approximately 
100 million or more; see \codelink{https://faculty.fordham.edu/dswinarski/MukaiModelOfM7/v1/6.2.htm}{6.2}.

Here is an observation that leads to a successful
approach. $\mywedge^7 S^{+}$ is reducible; we have 
 $\mywedge^7 S^{+} \cong V_1 \oplus V_2$, where 
 $V_1$ has highest weight $(1,0,1,0,1)$ and $V_2$ has highest weight
 $(3,0,0,1,0)$.  We have $\dim V_1=8800$ and $\dim V_2=2640$, and
 highest weight vectors $v_1$ and $v_2$ generating these modules are
 as follows; see
 \codelink{https://faculty.fordham.edu/dswinarski/MukaiModelOfM7/v1/6.3.htm}{6.3}. 
 \begin{align*}
   v_1 &= y_{\{1,2\}, \{1, 3\}, \{1, 2, 3, 4\},\{1, 2, 3, 5\}, \{1, 2, 4, 5\}, \{1, 3, 4, 5\}, \{2, 3, 4, 5\}}\\
   v_2 &= y_{\{1, 2\}, \{1, 3\}, \{1, 4\}, \{1, 2, 3, 4\}, \{1, 2, 3, 5\}, \{1, 2, 4, 5\}, \{1, 3, 4, 5\}}
 \end{align*}
Thus
  \[
\Sym^d(V_1 \oplus V_2) \cong \sum_{d_1+d_2=d} \Sym^{d_1} V_1  \otimes
\Sym^{d_2} V_2 
\]

We focus on the summand $\Sym^{2} V_1 \otimes \Sym^{2} V_2$. There are
89 $\Spin(10)$ invariants in this summand, and they are all of the form
\[
   ( V(\lambda) \otimes V(\lambda^*) )^{\Spin(10)}
 \]
for some irreducible $V(\lambda) \subset V_1$ with dual
$V(\lambda^{*}) \subset V_2$; see \codelink{https://faculty.fordham.edu/dswinarski/MukaiModelOfM7/v1/6.4.htm}{6.4}.

Next, we analyze the irreducible decompositions of $\Sym^{2} V_1$ and
$\Sym^{2} V_2$ and select one dual pair of summands for further
study.  Specifically, we select the summand of $V_1$ with highest weight $(5,0,0,0,0)$.
$\dim V( 5\omega_1) = 1782$.  The rationale for this choice is that,
on the one hand, if $\lambda$ is too far from 0 in the weight lattice,
$V(\lambda)$ will have large dimension, and the subsequent
calculations in $V(\lambda) \otimes V(\lambda^*) $ will be difficult. But if
$\lambda$ is too close to 0 in the weight lattice, then the weight
$\lambda$ and $\lambda^{*}$ spaces in $\Sym^{2} V_1$ and
$\Sym^{2} V_2$  will have large dimension, and it will be
difficult to compute highest weight vectors generating $V(\lambda)$
and $V(\lambda^{*})$. Choosing $\lambda = (5,0,0,0,0)$ was a
compromise between these competing considerations; see \codelink{https://faculty.fordham.edu/dswinarski/MukaiModelOfM7/v1/6.5.htm}{6.5}.

Next, observe that $V(5\omega_1)$ appears in the fifth symmetric power
of the standard representation of $\so(10)$; see \codelink{https://faculty.fordham.edu/dswinarski/MukaiModelOfM7/v1/6.6.htm}{6.6}.
 \[
  \Sym^5 \operatorname{Std}   \cong V( 5\omega_1) \oplus V( 3\omega_1)
  \oplus V(\omega_1)
\]

Choose an explicit basis of $V( 5\omega_1) \subset \Sym^5 \operatorname{Std}  $ consisting of
  elements of the form 
  \[
f_I  = X_{-\alpha_{i_k}} \ldots X_{-\alpha_{i_1}}.w
  \]
where $w$ is a highest weight vector of $V( 5\omega_1)$ and $I =
\{i_1,\ldots,i_k\}$ indexes a sequence of negative roots; see
\codelink{https://faculty.fordham.edu/dswinarski/MukaiModelOfM7/v1/6.7.htm}{6.7}.
This yields a basis $B_{5\omega_1} = \{ f_I \otimes g_J \}$ of $V( 5\omega_1) \otimes
V( 5\omega_1)$.

 The $T$-invariants of $V( 5\omega_1) \otimes
V( 5\omega_1)$ are spanned by the basis elements $f_I \otimes g_J$ in
which  $f_I$ and $g_J$ have opposite weight. We have  
\[
\dim(V( 5\omega_1) \otimes V( 5\omega_1) )^{T} = 4722;
\]
see \codelink{https://faculty.fordham.edu/dswinarski/MukaiModelOfM7/v1/6.8.htm}{6.8}.
The dimension of this space is sufficiently small that we can 
compute the kernel of the restriction of the Casimir operator $c$ to
this space using the iterative approach suggested in Proposition
\ref{prop:iterative strategy}. We obtain a symbolic expression for an
invariant polynomial that we denote $F_{5\omega_1}$.

\begin{proposition}
We have explicit lists of sequences $I$ and $J$ defining a basis of
$(V(5\omega_1) \otimes V( 5\omega_1))^{T}$ and coefficients $c_{IJ}
\in \mathbb{Q}$ such
that the linear combination 
\begin{equation} \label{eqn:F5w1}
F_{5\omega_1}= \sum_{I,J} c_{I J} (X_{-\alpha_{i_k}} \ldots X_{-\alpha_{i_1}}.w_1) \otimes (X_{-\alpha_{j_\ell}} \ldots X_{-\alpha_{j_1}}.w_2)
\end{equation}
is a $\Spin(10)$ invariant polynomial.
\end{proposition}
See \codelink{https://faculty.fordham.edu/dswinarski/MukaiModelOfM7/v1/6.9.htm}{6.9}.

One more ingredient is needed in order to evaluate this symbolic
expression for $F_{5\omega_1}$ on points $[P] \in \Gr(7,16)$: namely, we need highest weight vectors
$w_1$ and $w_2$  generating $V( 5\omega_1) \subset \Sym^{2} V_1$ and
$V( 5\omega_1) \subset \Sym^{2} V_2$, respectively. We obtain these as
follows. The Casimir
operator acts on $V(5 \omega_1)$ with eigenvalue $65$, and acts by
different scalars on the
other irreducible submodules of $\Sym^{2} V_1$ and $\Sym^{2} V_2$
containing the $5 \omega_1 $ weight space.  Thus, we can compute $w_1$ and $w_2$ by
iteratively projecting away the eigenspaces of $(\Sym^{2}
V_1)_{(5,0,0,0,0)}$ and $(\Sym^{2} V_2)_{(5,0,0,0,0)}$ corresponding
to the other eigenvalues of the Casimir operator $c$. This yields vectors $w_1$ and $w_2$
having 569 terms and 785 terms, respectively; see \codelink{https://faculty.fordham.edu/dswinarski/MukaiModelOfM7/v1/6.10.htm}{6.10}.

\textit{Remark.} We can consider $V(5 \omega_1) \subset \Sym^5
\operatorname{Std} $ and simplify the  expression (\ref{eqn:F5w1}) for
$F_{5\omega_1}$. This yields an $\SO(10)$-invariant polynomial of
bidegree $(5,5)$ in two sets of 10 variables. It has 7502
terms; see \codelink{https://faculty.fordham.edu/dswinarski/MukaiModelOfM7/v1/6.11.htm}{6.11}.
It seems likely that this polynomial has been described in the literature before,
but I do not know a reference for this.

\subsection{$\Spin(10)$-semistability of singular
  curves} \label{subsec:Spin10 semistability}

We now state and prove the main theorem.

\begin{theorem} \label{theorem:main}
The points $[P] \in \Gr(7,16)$ parametrizing the following singular curves are $\Spin(10)$-semistable.
\begin{enumerate}
\item The 7-cuspidal curve with heptagonal symmetry $\CCusp$
\item The genus 7 balanced ribbon $\CRib$
\item The reducible nodal curves $\CNod{t}$ for $t \neq 1, \infty$
\end{enumerate}
\end{theorem}

\begin{proof}
We use the linear spaces $\PCusp$, $\PRib$, and $\PNod{t}$ described in
Propositions \ref{prop:PCusp}, \ref{prop:PRib}, and \ref{prop:PNodt}.
  
We have
\begin{align*}
F_{5 \omega_1} (\PCusp) &= -63984375\\
F_{5 \omega_1} (\PRib) &= \frac{92664000}{343}\\
F_{5 \omega_1} (\PNod{t}) &= t_1^2  (t_0 - t_1)^3  \frac{234000}{343}\\
\end{align*}
See \codelink{https://faculty.fordham.edu/dswinarski/MukaiModelOfM7/v1/6.12.htm}{6.12}.

Since there exists a $\Spin(10)$ invariant polynomial that does not
vanish at these points, these points are $\Spin(10)$-semistable.
\end{proof}

Recall that by Proposition \ref{prop:GIT instability}, we know
that $\PNod{1}$ and $\PNod{\infty}$ are $T$-unstable, hence
$\Spin(10)$-unstable. Thus, we have a complete description of
$\Spin(10)$-semistability or instability for each member of the family
$\CNod{t}$. These results naturally suggest the following question.

\begin{question} What are the GIT semistable replacements for the family $\PNod{t}$ when $t=1$
and $t=\infty$?
\end{question}

Foundational references for the statement of GIT
semistable replacement include
\cite{Mumford}*{Lemma 5.3}, \cite{Seshadri}*{Theorem 4.1.i}, and
\cite{Shah}*{Proposition 2.1}. More recent
references include \cite{Caporaso}*{Section 1.2.1}, \cite{CM}*{Theorem
  11.1}, and \cite{Laza}*{Proposition 1.7}. Unfortunately, none of these
references give an effective algorithm for computing the GIT semistable replacement.

\textit{Remark.} The calculations reported in the proof of Theorem
\ref{theorem:main} required very large amounts of time and memory. They
were accomplished by parallel calculations on four AWS
\texttt{r5.24xlarge} instances, each with 96 vCPUs and 768 GB
memory. This took approximately 36 hours. In future work, we will try to improve the \texttt{Macaulay2} code for
these calculations to permit additional calculations at a lower cost.

\section*{References}
\bibliographystyle{amsplain}

\end{document}